\documentclass[onefignum,onetabnum]{siamart190516}
\usepackage{graphicx}
\usepackage{amsmath}
\usepackage{amssymb}
\usepackage{color}
\usepackage{verbatim}
\usepackage{tikz}
\usepackage{subfig}
\usepackage{url}
\usepackage{enumerate}
\usepackage{tabularx}
\usepackage[multiple]{footmisc}
\usepackage[utf8]{inputenc}
\usepackage{lineno}

\usetikzlibrary{shapes.multipart}

\DeclareMathOperator{\abs}{sabs}
\DeclareMathOperator{\sheavi}{sH}

\newtheorem{remark}{Remark}

\begin{document}

\title{Continuation method for PDE-constrained global optimization: Analysis and application to the shallow water equations}

\author{J. H. Baayen\thanks{KISTERS Nederland, St. Jacobsstraat 123-135, 3511 BP Utrecht, The Netherlands (\email{jorn.baayen@kisters-bv.nl})}
\and T. Piovesan\thanks{Deltares, Boussinesqweg 1, 2629 HV Delft, The Netherlands}
\and J. VanderWees\footnotemark[1]}

\maketitle

\begin{abstract}
This paper shows how a class of non-convex optimization problems constrained by discretized nonlinear partial differential equations may be solved to global
optimality using an interior point continuation method.  The solution procedure rests on a nested homotopy.  
The inner homotopy solves a barrier problem by driving the barrier parameter to zero. The outer homotopy deforms a convex relaxation
to the original non-convex problem in a way that stays clear of bifurcations.  A requirement for global optimality is that the objective is convex and that the search space remains
path-connected.  As a case study, a class of real-world optimization problems subject to the shallow
water equations is analyzed. A benchmark as well as a practical implementation demonstrate that the approach is suitable for closed-loop non-convex model predictive control
of large-scale cyber-physical systems.
\end{abstract}



\section{Introduction}

Optimization problems constrained by discretized nonlinear partial differential equations arise
in the context of numerical optimal control of cyber-physical systems, such as river systems
with hydropower cascades \cite[e.g.]{ackermann2000real,Schwanenberg2015}.  In general, these non-convex problems cannot be solved to global optimality by a naive application
of an interior point method.  They can, however, be solved to global optimality using polynomial hierarchies
\cite{lasserre2001global}, or using a homotopy continuation method that tracks all zeroes of a deforming system of polynomials \cite{bates2013numerically}.
Both of these methods suffer from high computational complexity and cannot be applied to large problems
in a closed-loop setting with tight limits on computation time.

In this paper, we look at the general homotopy continuation method \cite[e.g.]{allgower2012numerical} from a different angle.  Instead of tracking zeroes
of a polynomial as in \cite[e.g.]{bates2013numerically, mehta2016numerical}, we set up a homotopy between a convex relaxation and the original non-convex problem.  In this way,
the number of variables of the optimization problem does not increase (as they would with, e.g., a Lasserre hierarchy), and we 
may restrict our attention to the tracking of a single solution.  The resulting method is therefore readily applied to problems with a large number of variables\footnote{The benchmark in \cite{Baayen2019-3} covers
problems with up to $500\,000$ variables.}.


We will now give a brief overview of this method. Let $\theta \in [0,1]$ be the deformation parameter, where $\theta=0$ corresponds to the convex approximation of the non-convex optimization problem and $\theta=1$ to the original non-convex problem.
By construction, the approximated convex problem only admits global optima.  Let $x_{cp}$ denote such a global optimum, and let $\mathcal{S}$ denote the space of all possible solutions for all $\theta \in [0,1]$.  
We describe a procedure to find an optimal solution of each stage of the deformation, starting from $x_{cp}$.
That is, we construct a well-behaved ``problem-to-solution'' function $f:[0,1] \to \mathcal{S}$ where any $f(\theta)$ is an optimal solution to the corresponding optimization problem deformed by $\theta$ from $f(0) = x_{cp}$.
Here well-behaved is taken to mean that the function is continuous and does not contain any singularities.  Singularities would produce bifurcations and other undesired behavior \cite{poore1987bifurcation,guddat1990parametric}. These basic properties allow us to derive a method to find a solution for the non-convex problem, $f(1)$, starting from an optimal solution of the convex approximation, $f(0)$,
by tracing a uniquely defined path of solutions as $\theta$ is taken from zero to one.   These properties also allow us to prove that the solution at the end of the path, at $\theta=1$, is a global optimum.
In Section \ref{sec:def}, we formally describe this approach and provide sufficient conditions to ensure that the path exists, is unique, and that the problem does not admit any other solutions. 
The results hinge on two newly defined notions: \emph{zero-convexity} and \emph{path-stability}. 

In Section \ref{sec:application}, we consider an application of the homotopy method to the shallow water equations.
These equations occur when setting up decision support systems for river and canal systems, such as those managed by Rijnland water authority 
in the area around the city of Leiden and Amsterdam Schiphol Airport in the Netherlands (the area covers approximately $1 175$ km$^2$). At Rijnland, path-stable continuation is in day-to-day use for closed-loop model predictive control of 
$4$ primary pumping stations to control water levels and water quality in the primary canal system with a total length of approximately $370$ km \cite{Bosbo1, baayen2019overview}.


\section{Background}\label{sec:def}

A general continuous optimization problem can be formulated in the following standard manner:
\begin{align}
\min_x f(x) & \quad \text{subject to} \nonumber \\
c(x) & = 0, \tag{$\mathcal P$}\label{def:optimization-problem-bis}\\
x_i & \geq 0 \quad i \in \{1,\ldots,m\}, \nonumber \\ 
x &  \in \mathbb{R}^n. \nonumber
\end{align}
Note that we make a distinction between bounded, nonnegative variables and unbounded variables.

We assume throughout that all the functions are twice continuously differentiable.  We will refer to such functions as being (sufficiently) smooth. Let $f$ denote the \emph{objective function} of problem (\ref{def:optimization-problem-bis}).

\subsection{Interior point methods}

Interior point methods are used to find local minima of general optimization problems \cite{wright1997primal, renegar2001mathematical, forsgren2002interior, Wachter2006}. We will briefly mention some notions that we will need for our purposes. 
The general idea is to find a solution by computing (approximate) solutions for a sequence of barrier problems.
A barrier problem, for a parameter $\mu >0$, is defined as:
\begin{align}
\min_x f(x) - \mu \sum_{i=1}^m \ln x_i &  \quad \text{subject to} \nonumber\\
c(x) & = 0, \tag{$\mathcal P_{\mu}$}\label{def:barrier-optimization-problem}\\
x  & \in \mathbb{R}^n. \nonumber
\end{align}
This reformulation allows us to remove the non-negativity constraints on the variables.  As long as the algorithm starts with strictly positive $x_i$, $i \in \{1,\ldots,m\}$, the logarithmic barrier terms in the objective function will ensure that the solution coordinates remain
strictly positive.  Furthermore, if $f$ is a convex function and $c$ linear, then the reformulation also turns the convex optimization problem
(\ref{def:optimization-problem-bis}) into a strictly convex optimization problem (\ref{def:barrier-optimization-problem}) with a unique solution.
Generally speaking, for any sequence of barrier parameter values $\mu$ converging to zero, the sequence of the corresponding solutions to the problems (\ref{def:barrier-optimization-problem}) converges to a solution of (\ref{def:optimization-problem-bis}) (see \cite{forsgren2002interior}).

The objective function in the barrier problem (\ref{def:barrier-optimization-problem}) is only defined for \emph{interior points}:

\begin{definition}
Consider the barrier problem (\ref{def:barrier-optimization-problem}).  
A point $x$ is called a \emph{feasible interior point} if $x_i > 0$ for all $i \in \{1,\ldots,m\}$ and if it satisfies the constraints $c(x)=0$.
\end{definition}

For a generic optimization problem, the standard strategy to find a local minimum is to use the method of the Lagrange multipliers. 
The \emph{Lagrangian} of the problem (\ref{def:optimization-problem-bis}) is \cite{floudas1995nonlinear, geiger2013theorie}:
$$\mathcal L (x, \lambda) := f(x) + \lambda^T c(x)$$ where $\lambda$ is the vector of Lagrangian multipliers. 
Any local minimum of (\ref{def:optimization-problem-bis}) is a solution to the system of equations
\begin{eqnarray}\label{eq:primal}
\nabla_{x} \mathcal L (x, \lambda)  & = & 0, \label{eq:primal1}\\
c(x) & = & 0. \nonumber \label{eq:primal2}
\end{eqnarray}

\subsection{Parametric programming}

A parametric optimization problem is a particular type of optimization problem where the objective and constraint functions depend on a parameter $\theta \in [0,1]$:
\begin{align}
\min_x f(x,\theta) & \quad  \text{subject to}  \nonumber\\
c(x,\theta)  & = 0, \tag{$\mathcal P^{\theta}$}\label{def:parametric-optimization-problem}\\
x_i & \geq 0 \quad i \in \{1,\ldots,m\}, \nonumber \\ 
x &  \in \mathbb{R}^n, \nonumber
\end{align}
where $x$ is the optimization variable, $f(x,\theta)$ is the objective function and $c(x,\theta)$ denotes the constraints. 

The main idea of this paper is to continuously deform an optimization problem, such that ($\mathcal P^{\theta = 0}$) is a convex problem and ($\mathcal P^{\theta = 1}$) is the original non-convex problem, and track the corresponding solution $x^*(\theta)$.
From the previous discussion, we know that any solution of (\ref{def:parametric-optimization-problem}) is the solution of a system of equations. Thus, we can equivalently consider the continuous deformation of a system of equations and track its solution.

\subsection{Continuation methods}\label{sec:cont}

Here we provide a brief overview of the classical continuation method \cite{allgower2012numerical}. 
Let $F: \mathbb{R}^n \to \mathbb{R}^n$ denote the residual function for a system of nonlinear equations of the form
\[
F(x)=0.
\]
In general, finding a solution $x^*$ such that $F(x^*)=0$ is a hard problem. If an initial guess $x_0$ is sufficiently close to a solution and the function $F$ satisfies certain regularity properties, the Newton-Raphson method will converge to $x^*$. But if $x_0$ is too far away, the Newton-Raphson method may diverge and a different approach is needed.

The continuation method is one such approach and we will now sketch the idea behind it.  One approximates the residual function $F$ with a suitable function $\tilde{F}$, for which a solution $\tilde x^*$
is known:
\[
\tilde{F}(\tilde x^*) = 0.
\]
A parameter $\theta$ is then introduced to deform $\tilde{F}$ into $F$ using the homotopy
\begin{equation}\label{eq:simple-homotopy}
G(x,\theta):=(1-\theta)\tilde{F} + \theta F
\end{equation}
as $\theta$ goes from $0$ to $1$. 
With $\tilde x^*$ given such that $\tilde{F}(\tilde x^*)=G(\tilde x^*,0)=0$, we can increase $\theta$ and solve $G(x,\theta)=0$ for $x$ starting from $\tilde x^*$, which, if the increase in $\theta$ was sufficiently small, will lie sufficiently close to the solution of $G(x,\theta)=0$ for the Newton-Raphson method to converge. Continuing in this way, under suitable conditions, we arrive at a solution $x^*$ such that $F(x^*)=G(x^*,1)=0$.  

In the process, we have traced a path $\theta \mapsto x(\theta)$. By the implicit function theorem \cite[Theorem 9.28]{rudin1964principles}, this path exists locally and uniquely, and is continuously differentiable, as long as $\partial G / \partial x$ is nonsingular. 
This motivates the following definition:

\begin{definition}
Consider the homotopy (\ref{eq:simple-homotopy}).  
A point $x$ is called \emph{singular} if the Jacobian matrix $\partial G / \partial x$ is singular at $x$.
\end{definition}

At singular points, the path may (1) turn back on itself, (2) end, or (3) bifurcate into multiple paths. Figure~\ref{fig:1} illustrates a path with a bifurcation, and highlights the point where $\partial G / \partial x$ is singular.  Clearly, this situation is undesirable and in the following we will look for conditions under which all points are nonsingular.

\begin{figure}
\centering
  \includegraphics[scale=0.5]{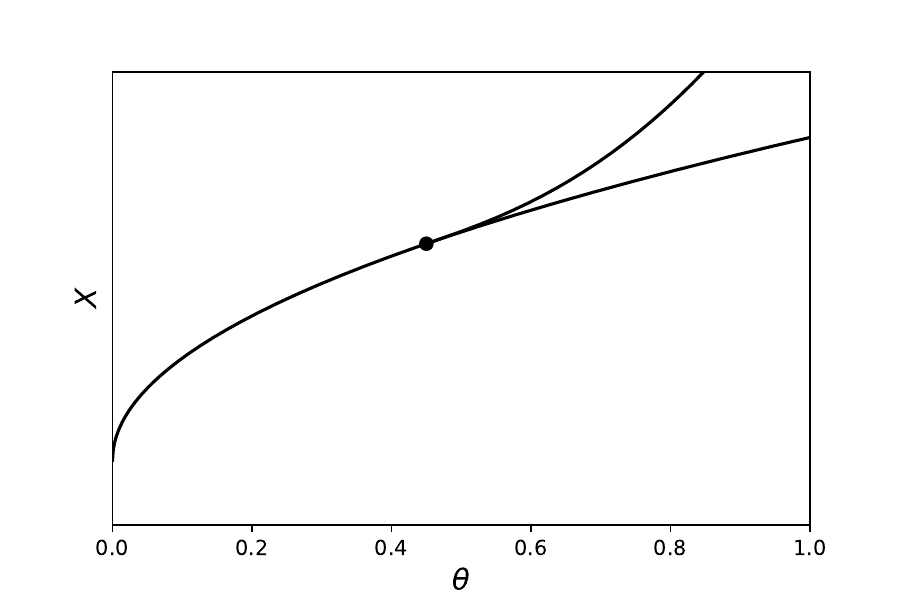}
\caption{A path with a bifurcation. The singular point is highlighted.}
\label{fig:1}       
\end{figure}

\section{Continuation method for global optimization}


Our objective is to construct, for any $\mu > 0$ and any $\theta \in [0,1]$, an optimization problem (\ref{def:parametric-optimization-problem}) such that we can track its solution
using a continuation method. 
Using interior point methods, a problem (\ref{def:parametric-optimization-problem}) is solved by equivalently finding the solution of a system of equations (\ref{eq:primal1}).
Continuation theory guarantees that the path of the solution of the system of equations does not bifurcate (i.e., it can be traced) if the Jacobian of the residual function is not singular.

More formally, we want to track the solution of the system of equations 
\begin{eqnarray}\label{eq:primal-bar}
\nabla_{x} \mathcal L_\mu (x, \lambda, \theta)  & = & 0, \label{eq:primal1-bar}\\
c(x,\theta) & = & 0. \nonumber
\end{eqnarray}
where $\mathcal L_\mu (x, \lambda, \theta) := f(x, \theta) - \mu \sum_{i=1}^m \ln x_i + \lambda^T c(x,\theta)$ is the Lagrangian of the following parametric barrier problem:
\begin{align}
\min_x f(x, \theta) - \mu \sum_{i=1}^m \ln x_i &  \quad \text{subject to} \nonumber\\
c(x,\theta) & = 0, \tag{$\mathcal P_{\mu}^{\theta}$}\label{def:parametric-barrier-optimization-problem}\\
x  & \in \mathbb{R}^n. \nonumber
\end{align}
Let $F_\mu(x,\lambda,\theta)$ denote the residual of the system of equations (\ref{eq:primal-bar}). Then the system
\[
F_\mu(x,\lambda,\theta)=0
\]
admits a unique solution path in a neighborhood of $x^*$, $\lambda^*$, and $\theta^*$ if the Jacobian matrix $\partial F_\mu / \partial (x,\lambda)$ is nonsingular at the point $(x^*,\lambda^*,\theta^*)$.

In the following, we will need the notion of the tangent space of the constraint manifold:
\begin{definition}[e.g., \cite{poore1987bifurcation}]
Fix a $\theta \in [0,1]$ and a feasible interior point $x$. We call the linear space
\[
T(x,\theta):=\{y : \nabla_x c(x,\theta) y =0\}
\]
the \emph{tangent space} of the constraint set $c(x,\theta)=0$ at $x$.
\end{definition}

\subsection{Sufficient conditions for convergence to a global optimum}

We are now ready to discuss sufficient conditions for the path tracing procedure to converge to a global optimum.  To do so, we will need to introduce two new notions: \emph{zero-convexity} and \emph{path-stability}.

\begin{definition}\label{def:zero-convex}
We say that the parametric optimization problem (\ref{def:parametric-optimization-problem}) is \emph{zero-convex} if the objective function $x \mapsto f(x,0)$ is a convex function, and the constraints $x \mapsto c(x,0)$ are linear.
\end{definition}

The notion of zero-convexity captures the idea that there should be a unique solution at $\theta=0$, and that it should be possible to find this solution using standard methods.  

\begin{definition}\label{def:path-stable}
We say that the parametric optimization problem (\ref{def:parametric-optimization-problem}) is \emph{path-stable with respect to the interior point method}
if its barrier formulation (\ref{def:parametric-barrier-optimization-problem}) does not admit singular feasible points for any $\mu > 0$ and any $\theta \in [0,1]$.
\end{definition}

The concept of \emph{path-stability} captures the idea that we seek a way to consistently arrive at a uniquely related local minimum of the fully nonlinear problem at $\theta=1$, i.e., 
without path bifurcations along the way.  

The following Proposition provides an important characterization of path-stability:

\begin{proposition}\label{theorem:tiahrt}
Consider the parametric optimization problem (\ref{def:parametric-barrier-optimization-problem}).  
Fix a $\mu > 0$ and a $\theta \in [0,1]$. Let $x$ denote a feasible point.
The point $x$ is nonsingular if and only if the following two conditions hold:
\begin{enumerate}
\item The Hessian matrix $\nabla_{xx}^2 \mathcal{L}_{\mu}(x,\lambda,\theta)$ is nonsingular on the tangent space $T(x,\theta)$;
\item The Jacobian matrix of the constraints $\nabla_x c(x,\theta)$ has full rank. 
\end{enumerate}
\end{proposition}
\begin{proof}
Consider the Jacobian matrix
\[
\frac{\partial F_\mu}{\partial (x,\lambda)}=\nabla^2_{xx} \mathcal{L}_{\mu}(x,\lambda,\theta).
\]
Define a local basis $e_1,\ldots,e_\ell$ that spans the tangent space $T(x,\theta)$, and a basis $e_{\ell+1},\ldots,e_n$ that spans
its orthogonal complement.  With respect to these bases, the Jacobian matrix has the form
\[
\begin{pmatrix}
\nabla_{e_{1},\ldots,e_\ell;e_{1},\ldots,e_\ell}^2 \mathcal{L}_{\mu} & \nabla_{e_{1},\ldots,e_\ell;e_{\ell+1},\ldots,e_n}^2 \mathcal{L}_{\mu} & 0 \\
\nabla_{e_{1},\ldots,e_\ell;e_{\ell+1},\ldots,e_n}^2 \mathcal{L}_{\mu} & \nabla_{e_{\ell+1},\ldots,e_n;e_{\ell+1},\ldots,e_n}^2 \mathcal{L}_{\mu} & \nabla^T_{e_{\ell+1},\ldots,e_n} c \\
0 & \nabla_{e_{\ell+1},\ldots,e_n} c & 0 \\
\end{pmatrix},
\]
since $d c(x,\theta) / d e_i=\nabla_x c(x,\theta) e_i=0$ if $e_i \in T(x,\theta)$.
If (2) does not hold, then $\nabla_{e_{\ell+1},\ldots,e_n} c$ cannot have full row rank, whence the Jacobian matrix must be singular.
If (2) holds, $\nabla_{e_{\ell+1},\ldots,e_n} c$ is a nonsingular square matrix so that elementary row and column operations can transform the Jacobian matrix to
a matrix of the form
\[
\begin{pmatrix}
\nabla_{e_{1},\ldots,e_\ell;e_{1},\ldots,e_\ell}^2 \mathcal{L}_{\mu} & 0 & 0 \\
0 &0  & \nabla^T_{e_{\ell+1},\ldots,e_n} c \\
0 & \nabla_{e_{\ell+1},\ldots,e_n} c & 0 \\
\end{pmatrix},
\]
which is nonsingular if and only if condition (1) holds.
We have shown that the Jacobian matrix is nonsingular if and only if conditions (1) and (2) hold.
\end{proof}

\begin{remark}
Condition (1) in Proposition \ref{theorem:tiahrt} is equivalent to the condition that the \emph{reduced Hessian} \cite[e.g.]{fletcher2013practical, geiger2013theorie} 
\[
Z(x,\theta)^T \nabla_{xx} \mathcal{L}_{\mu}(x,\lambda,\theta) Z(x,\theta)
\]
is positive definite, with $Z(x,\theta)$ a basis matrix of $T(x,\theta)$.
\end{remark}

\begin{remark}
Condition (2) in Proposition \ref{theorem:tiahrt} is also known as the \emph{linear independence constraint qualification} (LICQ) \cite{floudas1995nonlinear, geiger2013theorie}.
\end{remark}

\begin{remark}
For an example of a problem where the Hessian of the Lagrangian is singular on the tangent space, consider the two-dimensional toy problem: $\min_x x_1^2+x_2^2$ subject to the constraint $x_1^2+x_2^2=1$.  Note how every feasible
point of the toy problem is a non-strict local minimum.
\end{remark}

\begin{lemma}\label{lemma:posdef}
Fix a $\mu > 0$ and a $\theta \in [0,1]$.  If the parametric optimization problem (\ref{def:parametric-optimization-problem}) is zero-convex and path-stable with
respect to the interior point method, then the Hessian matrix 
\[
\nabla_{xx}^2 \mathcal{L}_{\mu}(x,\lambda,\theta)
\]
is positive definite on the tangent space $T(x,\theta)$ for all feasible interior points $x$, and all $\lambda$.
\end{lemma}

\begin{proof}
Fix a $\mu > 0$. For $\theta=0$, path-stability implies that $\nabla_{xx}^2 \mathcal{L}_{\mu}(x,\lambda,0)$ is nonsingular
on the tangent space $T(x,\theta)$ (Proposition \ref{theorem:tiahrt}).  Therefore, the eigenvalues
of $\nabla_{xx}^2 \mathcal{L}_{\mu}(x,\lambda,0)$ on the tangent space cannot be zero.
Zero-convexity, together with the strict convexity of the barrier terms and the linearity of the constraints,
implies that the eigenvalues of $\nabla_{xx}^2 \mathcal{L}_{\mu}(x,\lambda,0)$ (which are real as per the spectral theorem for symmetric matrices) cannot be negative.
Therefore, the matrix $\nabla_{xx}^2 \mathcal{L}_{\mu}(x,\lambda,0)$ must be positive definite on the tangent space $T(x,\theta)$.

For $\theta > 0$, note that the eigenvalues -- being the roots of the characteristic polynomial -- vary continuously with $\mu$, $\theta$,
$x$, and $\lambda$ \cite{harris1987shorter} since $\mathcal{L}_{\mu}$ is twice continuously differentiable.  As such, a negative eigenvalue
can only arise if there would exist a $\theta > 0$ and $(x,\lambda)$ such that $\nabla_{xx}^2 \mathcal{L}_{\mu}(x,\lambda,\theta)$
would have a zero eigenvalue on the tangent space $T(x,\theta)$, and hence would be singular there.  But this would contradict the assumption of path-stability
due to Proposition \ref{theorem:tiahrt}.
\end{proof}

\begin{remark}
Lemma \ref{lemma:posdef} shows that the reduced Hessian of a zero-convex and path-stable problem has strictly positive
eigenvalues for \emph{all} feasible interior points.  This is a stronger condition than the usual second-order sufficiency condition that the reduced
Hessian be positive definite at stationary points of the Lagrangian. The stronger condition is used to prove uniqueness in Theorem \ref{prop:uniqueness}.
\end{remark}

In order to describe the sufficient conditions that ensure the existence of no more than one solution, we will need to recall the notion of \emph{path-connected} set:

\begin{definition}[e.g., \cite{mendelson1962introduction}]
A set $X$ is \emph{path-connected} if for any $x_1, x_2 \in X$, there exists a continuous function $f: [0,1] \to X$
such that $f(0)=x_1$ and $f(1)=x_2$.
\end{definition}

The following Theorem describes sufficient conditions for our parametric optimization problem to have at most one solution, so that
any solution must be the global optimum:
\begin{theorem}\label{prop:uniqueness}
Consider the parametric barrier problem (\ref{def:parametric-barrier-optimization-problem}). Assume that
\begin{enumerate}
\item the problem is zero-convex;
\item the problem is path-stable;
\item the set of feasible interior points is path-connected.
\end{enumerate}
Then for any $\mu > 0$ and any $\theta \in [0,1]$, the barrier problem (\ref{def:parametric-barrier-optimization-problem}) has at most one unique solution. 
This unique solution is its global optimum.
\end{theorem}
\begin{proof}
Fix $\mu > 0$ and $\theta \in [0,1]$. Suppose that the system of equations $F_\mu(x,\lambda,\theta)~=~0$ admits two different solutions.
From Lemma \ref{lemma:posdef} and the second order sufficiency conditions \cite[e.g.]{fletcher2013practical,geiger2013theorie}, it follows that these are strict local minima.  
Connect the two strict local minima with a continuous path of feasible points.
The objective function is continuous, and the image of the path forms a compact set.  Therefore, by the extreme value theorem \cite{rudin1964principles}, 
the objective function must attain a maximum somewhere on the path.  Since the endpoints of the path are strict local minima,
the local maximum must lie in the interior of the path, and it must yield an objective value exceeding the objective values for the two local minima. But the existence of such a local maximum on the path would contradict the positive definiteness of the reduced Hessian matrix
implied by Lemma \ref{lemma:posdef}.
\end{proof}

For problems that satisfy the sufficient conditions of Theorem \ref{prop:uniqueness}, the \emph{central path} \cite{renegar2001mathematical} of the interior point method is uniquely defined.

To solve an optimization problem (\ref{def:optimization-problem-bis}) that has an associated parametric optimization problem (\ref{def:parametric-barrier-optimization-problem})
that satisfies the conditions of Theorem \ref{prop:uniqueness},
we provide a feasible starting point, and then let the interior point method implementation
drive $\mu \to 0$, while ensuring that at every iteration $\mu > 0$. \cite{wright1997primal, harris1987shorter, Wachter2006}.
For every $\mu$, local search is used to solve the associated barrier problem.

Note that the requirement to start with a feasible starting point is not restrictive.  This starting point can be obtained
either from a simulation computation prior to the optimization run, or alternatively, it can be obtained using a continuation algorithm.
The continuation algorithm in turn may be seeded using the solution of the problem at $\theta=0$, which is a convex problem that may be solved using 
an interior point method without a starting point \cite{renegar2001mathematical}.

\begin{remark}
Even if the search space is not path-connected, zero-convex and path stable optimization problems can
be solved to local optimality using the continuation method.  Due to path-stability, the homotopy path
cannot bifurcate, and hence the local optimum of the non-convex optimization problem is uniquely defined
by the global optimum of the convex problem at $\theta=0$.  
\end{remark}

In applications to concrete classes of problems, most of the difficulty resides in the proof of the linear independence of the constraint gradients throughout the deformation process.  While such analyses may be lengthy, it is worth pointing out that the desired linear independence typically is a necessary condition for the numerical integration of the dynamics.  Results on linear independence are therefore typically readily available for the non-linear dynamics at $\theta=1$. Such results are typically phrased in terms of the non-singularity of the Jacobian of the discretized 
dynamics \cite[e.g]{casulli1998conservative}, and only need to be adjusted to account for the parametric deformation of the model.  

The power of the path-stable continuation approach resides in the fact that the positivity of the eigenvalues of the reduced Hessian is directly carried over from the convex relaxation. In this way, we bypass the need to analyze the complete spectrum of the reduced Hessian of the non-convex problem. Since the direct analysis of the spectrum is, in general, hard, authors typically sidestep the issue by approximating non-convex problems using polynomial hierarchies \cite[e.g]{ghaddar2017polynomial} or linear models \cite[e.g]{horvath2018categorization}.  In this paper we show how path-stable continuation opens the road to \emph{direct} global optimization of a class of PDE-constrained non-convex optimization problems.

In the next section, we look at an application to problems constrained by the Saint-Venant equations. 
Applications to other PDE, such as the Hazen-Williams equations for pressurized flow, or the isothermal Euler equations governing the flow of gas in pipes \cite{hante2017challenges}, would follow a similar scheme.

\section{Application to the shallow water equations}\label{sec:application}

In the present section we prove that, under mild technical conditions, the sufficient conditions for global optimality hold when considering the one-dimensional shallow water equations.
We also illustrate how these technical conditions are readily satisfied in practice, and therefore how water optimization problems may be formulated with a homotopy free of bifurcations.  

In Subsection \ref{sec:dynamics}, we describe the one-dimensional shallow water equations, linear approximations to the equations, and explain their discretization.
In Subsection~\ref{sec:analysis}, we prove that our discretization of the shallow water equations satisfies the notions of zero-convexity and path-stability, and can be included in an optimization problem in such a way that the search space remains path-connected. The proofs use well-known results from real analysis, linear algebra, and general topology.
In Subsection \ref{sec:example}, we consider a numerical example where we solve an optimization problem subject to the discretized shallow water equations.

\subsection{The shallow water equations}\label{sec:dynamics}

In the present section, we summarize the one-dimensional shallow water equations.  These are also known as the \emph{Saint-Venant} equations \cite{de1871theorie}.

The shallow water equations describe situations in fluid dynamics where the horizontal length scale is large compared to the water depth.  The Saint-Venant equations are given by the momentum equation
\begin{equation}\label{eq:saint-venant}
\frac{\partial Q}{\partial t} + \frac{\partial}{\partial x}\frac{Q^2}{A} + gA\frac{\partial H}{\partial x} + g\frac{Q|Q|}{A R C^2}=0,
\end{equation}
with longitudinal coordinate $x$, time $t$, discharge $Q$, water level $H$, cross section $A$, hydraulic radius $R:=A/P$, wetted perimeter $P$, Ch\'{e}zy friction coefficient $C$, gravitational constant $g$, and by the mass balance (or continuity) equation
\begin{equation}\label{eq:mass-balance}
\frac{\partial Q}{\partial x} + \frac{\partial A}{\partial t} = 0.
\end{equation}

The cross section $A: H \mapsto A(H)$ and wetted perimeter $P: H \mapsto P(H)$ are twice continuously differentiable functions such that for all $H$, it holds that $A > 0$, $dA / dH > 0$, $d^2 A / dH^2 \geq 0$, $P > 0$ and $dP / dH > 0$.
The conditions $dA / dH > 0$ and $dP / dH > 0$ state that the cross sectional area and the wetted perimeter are strictly increasing functions
of the water level, and the condition $d^2 A / dH^2 \geq 0$ states that the channel width $dA / dH$ is a non-decreasing function of the water level.

For our purposes, we require that these functions must be defined for \emph{all} $H$, i.e., including $H < H_b$.  We do this
in order to be able to produce ``imaginary'' solutions where $H < H_b$.  This construction improves the topology 
of the search space (cf. Corollary \ref{lemma:path-connected}), eventually leading to the global
optimality result.  Note that in Subsection \ref{sec:analysis}, we will discuss
ways to impose ``soft'' constraints on water levels.

A simple approach to building such functions is to let $A$ and
$P$ approach their natural $H=H_b$ values asymptotically as $H \to -\infty$, and to extrapolate smoothly 
as $H \to \infty$.  Such functions can be set up by fitting a cubic B-Spline \cite{de1972calculating} 
to bathymetry data over the range of physically feasible water levels (perturbing $A(H_b)$ away from zero if necessary), and flanking the B-Spline fits 
with the appropriate smooth extrapolations.

In practice, however, it is typically not required to set up such extrapolations.  In Subsection \ref{sec:analysis}, we will show that
if a solution is found to the optimization problem, this solution must be a globally optimal solution.  It is easy to check,
a posteriori, whether a solution satisfies $H > H_b$ everywhere.  If a soft lower bound is set on $H$ (cf. Subsection \ref{sec:analysis}), an optimum
with $H \leq H_b$ can only arise if channel reaches fall dry due to a lack of water.

In the remainder of this paper, we will restrict our attention to smooth, subcritical solutions of the Saint-Venant equations.  
Correct handling of supercritical phenomena requires additional attention as discussed in, e.g., \cite{Stelling2003}.

\subsection{A linear approximation to the shallow water equations}

The mass balance equation (\ref{eq:mass-balance}) and the momentum equation (\ref{eq:saint-venant}) are both, in general, nonlinear.
The inclusion of these equations as equality constraints in an optimization problem 
results in a problem that is non-convex.
In the present section, we develop linear approximations of these equations.

Starting from a globally optimal solution of the convex optimization problem subject to the linear approximation of the Saint-Venant equations, we may use the continuation method to find a solution of the nonlinear problem.
In Subsection \ref{sec:analysis} we will show that this solution is the only solution, and hence the global optimum of the barrier formulation of the nonlinear problem.

We start by defining a \emph{nominal} water level $\overline{H}$.  Typically, this level would correspond to a mean water level or a level setpoint.
We obtain a linear approximation to the mass balance equation by considering a rectangular cross section with nominal width $\overline{w}:=(dA/dH)(\overline{H})$:
\begin{equation}\label{eq:linear-mass-balance}
  \frac{\partial Q}{\partial x} + \overline{w}\frac{\partial H}{\partial t}=0.
\end{equation}

We now turn our attention towards the momentum equation.  The water level gradient $\partial H/\partial x$ is a primary driver of the flow and the direction thereof.  In order to maintain directional variability in the linear model we, therefore, need to retain the water level gradient as-is. 
Hence, we linearize the pressure term around $\partial H / \partial x=0$ and $A=\overline{A}$ with $\overline{A}:=A(\overline{H})$, and obtain the linearized pressure term $g\overline{A}\partial H / \partial x$.

The quadratic nature of the friction term cannot be maintained in a linear model.  We apply the nominal cross section $\overline{A}$, which results in a nominal hydraulic radius $\overline{R}=\overline{A}/\overline{P}$ with $\overline{P}:=P(\overline{H})$, and linearize around $Q=\overline{Q}$.  The choice of $\overline{Q}$ does \emph{not} express a preferred flow direction due to the presence of the absolute value function. 

The convective acceleration term $\partial \left(Q^2/A\right) / \partial x$ is of limited significance in subcritical river wave propagation scenarios \cite{montero2013simplified},
and it turns out that we can show path-stability if we leave it out of the linear approximation.  This is the same approximation that is used to derive the so-called \emph{inertial wave} equations.
We obtain the following linear approximation to the momentum equation:
\begin{equation}\label{eq:linear-inertial-wave}
\frac{\partial Q}{\partial t} + g\overline{A}\frac{\partial H}{\partial x} + g\frac{Q|\overline{Q}|}{\overline{A} \overline{R} C^2}=0.
\end{equation}


\subsection{Semi-implicit discretization on a staggered grid}
\label{sec:homotopic-discretization}

We discretize our hydraulic equations on a staggered grid and semi-implicitly in time, analogous to the approaches set out in, e.g., \cite{casulli1990semi, casulli1998conservative, Stelling2003}.  The pressure term is discretized semi-implicitly in the sense that the levels are evaluated at time $t_j$, whereas the cross section is evaluated at time $t_{j-1}$.  The friction term is discretized semi-implicitly in the sense that the discharge $Q$ is evaluated at time $t_j$, whereas the cross section and hydraulic radius are evaluated at time~$t_{j-1}$.  The convective acceleration term is discretized explicitly in time.
The von Neumann stability of such semi-implicit discretizations is analyzed in, e.g., \cite{casulli1994stability}.

In the following, we will refer to those variables which lie between two other hydraulic variables as \emph{interior variables}.  All other hydraulic variables are referred to as \emph{boundary variables}.  The staggered grid, and the distinction between interior and boundary variables, is illustrated in Figure~\ref{fig:grid}. 

\begin{figure}[!t]
\centering
\begin{tikzpicture}
\draw (0,0) -- (4.5,0);
\foreach \x in {1,2,3,4,5}
    \draw (\x cm - 1 cm,1 pt) -- (\x cm - 1 cm,-1 pt) node[anchor=north] { \begin{tabular}{c} $H_\x$ \end{tabular} };
\foreach \x in {1,2,3,4}
    \draw (\x cm - 0.5 cm,1 pt) -- (\x cm - 0.5 cm,-1 pt) node[anchor=south] {\begin{tabular}{c} $Q_\x$ \end{tabular}};
\draw (4.5 cm,1pt) -- (4.5 cm,-1pt) node[anchor=south] {\begin{tabular}{c} $Q_5$ \end{tabular}};
\draw[<->] (0,1 cm) -- (0.5 cm,1 cm) node[anchor=south] {$\Delta x$} -- (1 cm, 1 cm);
\draw[->] (0,-0.75 cm) -- (4.5 cm,-0.75 cm) node[anchor=west] {$x$};
\end{tikzpicture}
\caption{Staggered grid with an upstream level boundary and a downstream flow boundary. Here, $H_1$ and $Q_5$ are boundary variables, whereas all other variables are internal.}
\label{fig:grid}
\end{figure}
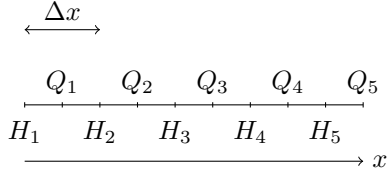

Throughout the paper we assume, without loss of generality, that the grid nodes are numbered as in Figure \ref{fig:grid}. That is, every interior variable $H_i$ has the variables $Q_{i-1}$ and $Q_i$, respectively, to its left and to its right. Such variables exist by construction. Similarly, any interior variable~$Q_i$ has the variable~$H_i$ to its left and $H_{i+1}$ to its right. 

We now introduce the homotopy parameter $\theta$ interpolating between the linear and nonlinear equations.  Following interpolation of the linear and nonlinear mass balance equations, (\ref{eq:linear-mass-balance}) and (\ref{eq:mass-balance}), respectively, and discretization on our staggered grid, we obtain the discretized homotopic mass balance equation
\begin{align}\label{eq:discretized-mass-balance}
c_{i,j} &:= \frac{Q_{i}(t_{j})-Q_{i-1}(t_{j})}{\Delta x} \\
& + \theta \frac{A_i(H_i(t_{j})) - A_i(H_i(t_{j-1}))}{\Delta t} + (1 - \theta) \overline{w} \frac{H_i(t_{j}) - H_i(t_{j-1})}{\Delta t}  = 0 \nonumber \\
& \quad \forall i \in I_H \quad \forall j \in \{1,\ldots,T\} \nonumber
\end{align}
with the index set $I_H$ such that every $H_i$, $i \in I_H$, is an interior variable.  This is a mildly nonlinear, mass-conservative formulation
as in \cite{casulli1998conservative}.

We now turn our attention to the momentum equation.  Interpolating between the linear
and nonlinear momentum equations 
(\ref{eq:linear-inertial-wave}) and~(\ref{eq:saint-venant}), respectively, and discretizing on our staggered grid, we obtain
the discretized homotopic momentum equation
\begin{align}\label{eq:discretized-momentum}
d_{i,j} & := \frac{Q_i(t_{j}) - Q_i(t_{j-1})}{\Delta t} + \theta e_{i,j} \\ 
& + g\left(\theta A_{i+\frac{1}{2}}(t_{j-1}) + (1- \theta) \overline{A}\right) \frac{H_{i+1}(t_{j}) - H_{i}(t_{j})}{\Delta x}  \nonumber \\
& + g\left( \theta \frac{P_{i+\frac{1}{2}}(t_{j-1}) \abs Q_i(t_{j-1})}{A_{i+\frac{1}{2}}(t_{j-1})^2} + (1-\theta) \frac{\overline{P} \abs \overline{Q}}{\overline{A}^2} \right) 
\frac{Q_i(t_j)}{C_i^2} = 0 \nonumber \\
& \quad \forall i \in I_Q \quad \forall j \in \{1,\ldots,T\} \nonumber
\end{align}
with 
\begin{eqnarray*}
A_{i+\frac{1}{2}}(t_j) & := & \frac{1}{2} \left( A_i(H_i(t_j)) + A_{i+1}(H_{i+1}(t_j)) \right); \\
P_{i+\frac{1}{2}}(t_j) & := & \frac{1}{2} \left( P_i(H_i(t_j)) + P_{i+1}(H_{i+1}(t_j)) \right),
\end{eqnarray*}
convective acceleration $e_{i,j}$, and the index set $I_Q$ such that every $Q_i$, $i \in I_Q$, is an interior variable.  The parameter $C_i$ indicates the local friction coefficient, and $H^b_{i}$ indicates the local bottom level.  In order to avoid singular derivatives, we use 
\[
\abs x:=\sqrt{x^2 + \varepsilon},
\]
where $\varepsilon$ is a small constant, as a smooth approximation for $|x|$.  Note that $\varepsilon$ can be taken arbitrarily small, thereby approximating the absolute value function to arbitrary accuracy.

Note that we have used a single set of constant nominal values $\overline{w}$, $\overline{A}$, $\overline{P}$, and $\overline{Q}$ for the entire reach.  
This is sufficient for the development of the theory.

The convective acceleration term $e_{i,j}$ must be discretized explicitly in time in order to be able to prove path-stability.  In the following, we consider
a finite difference approximation to the convective acceleration term, in which the finite differences are taken in the upstream
direction (a so-called \emph{upwind} scheme).  In order to ensure a smooth formulation regardless of the flow direction, we replace the Heaviside function with a logistic function,
finally obtaining
\begin{align*}
e_{i,j} := \sheavi(Q_i(t_{j-1})) \frac{2Q_i(t_{j-1})}{A_{i+\frac{1}{2}}(t_{j-1})} \frac{Q_i(t_{j-1})-Q_{i-1}(t_{j-1})}{\Delta x} \\ 
+ \left(1 - \sheavi(Q_i(t_{j-1}))\right) \frac{2Q_i(t_{j-1})}{A_{i+\frac{1}{2}}(t_{j-1})} \frac{Q_{i+1}(t_{j-1})-Q_{i}(t_{j-1})}{\Delta x} \\
- \frac{Q_i(t_{j-1})^2}{A_{i+\frac{1}{2}}(t_{j-1})^2} \frac{A_{i+1}(H_{i+1}(t_{j-1}))-A_i(H_i(t_{j-1}))}{\Delta x},
\end{align*}
with the logistic function
\[
\sheavi(x):= \frac{1}{1+e^{-Kx}},
\]
and steepness factor $K > 0$.  Note that $K$ can be taken arbitrarily large, thereby approximating the Heaviside function to arbitrary accuracy.

\subsection{Homotopy convergence analysis}\label{sec:analysis}

We consider a numerical optimal control problem subject to the dynamics (\ref{eq:discretized-mass-balance}) - (\ref{eq:discretized-momentum}) imposed as equality constraints between flow variables $Q$ and water level variables $H$.
Let $F_{\mu}(x,\lambda,\theta)=0$, where the vector $x$ contains the variables $Q$, $H$ denote the primal equation system (\ref{eq:primal}) corresponding to this optimization problem.
In particular, we denote with $x_{hyd}$ the vector of the interior hydraulic variables. 

In the following, we will show zero-convexity, path-stability, and path-connectedness of the search space for this type of problem, provided that the following assumptions hold:

\begin{itemize}
\item [BND] None of the interior hydraulic variables are bounded. All free boundary variables have both a lower bound as well as an upper bound such that the lower bound is strictly less than the upper bound.
\item [ICO] Initial values $Q_i(t_0)$ and $H_i(t_0)$ are provided and replaced into the model so that the variables at $t_0$ are no longer included in the optimization problem.
\item [HBC] Any water level boundary conditions are fixed, i.e., if $H_i$ is a water level boundary, then $H_i(t_j)=v_j$ for some time series $\{v_j\}_{j \in \{0,\ldots,T\}}$.  Furthermore, the values $v_j$ are replaced into the model so that the variables $H_i(t_j)$ are no longer included in the optimization problem.
\item [QBC] There is at least one free flow boundary condition.  Any two free flow boundary conditions must have at least one interior flow variable situated in between.
\item [OBJ] The objective function is twice continuously differentiable and convex.
 \end{itemize}

Condition BND is trivially satisfied.  If needed, terms may be included in the objective function
that penalize deviation from desired flow and level ranges.  The condition BND is required to ensure
that the search space remains path-connected.

The condition ICO may be satisfied by providing a complete initial state, 
possibly computed using a state estimation algorithm prior to the optimization run. 

The condition HBC, requiring the water level variables at the boundaries to have fixed values, is hardly restrictive.
A downstream water level variable only occurs in the momentum equation for the adjacent flow variable.
A free downstream level therefore translates to a ``free'' downstream discharge.
The downstream level variable may therefore be omitted, resulting in the adjacent flow variable becoming the new free
boundary variable.  There is no requirement for $Q$ boundaries to be fixed.

The condition QBC is trivially satisfied.  If two free flow boundaries would not have an interior flow variable
situated in between, then we would only be imposing the continuity equation, but not the momentum equation.

Conditions ICO, HBC, and QBC ensure that the constraints remain linearly independent for all $\theta \in [0,1]$. These conditions therefore
help ensure (together with OBJ) that we stay clear of bifurcations.

The condition OBJ states that the objective of optimization problem must be convex.
This is not restrictive in the sense that all standard convex objectives, such as minimization in the $1$, $2$, or $\infty$ norms, are allowed.


The remainder of this section is dedicated to proving that, under the assumptions mentioned above, the problem is zero-convex (Proposition~\ref{lemma:zero-convex}), path-stable (Proposition~\ref{thm:main-result}) and its feasible solutions are path-connected (Corollary~\ref{lemma:path-connected}). These three properties combined will allow us to deduce that the non-convex optimization problem has a unique global optimum (Theorem~\ref{thm}).
Moreover, as is discussed in the following Remark, this methodology is able to find all solutions of interest.

\begin{remark}\label{remark:degeneracy}
Suppose that the original non-convex optimization problem, prior to its transformation to a barrier formulation, has a solution. Then this solution is of one of the following two types:
\begin{enumerate}
\item an interior point; or
\item a point with any number of flow boundary bounds active.
\end{enumerate}
Points of type (1) can be reached by a sequence of interior points such that the objective function values
of the original and the barrier problems converge as $\mu \to 0$.  Points of type (1) therefore cannot obtain lower objective values
than those reached using the interior point method.

Points of type (2) can also be reached by a sequence of interior points.  
By Lemma~\ref{lemma:span}, and the implicit function theorem, there exists a neighborhood around the boundary variables of the solution  
for which the constraint manifold is defined (here we temporarily disregard the active bounds, which are arbitrary from the point of view of the dynamics).  This means that we can perturb away from the bounds into the interior,
and construct a sequence of interior points that converges to the solution.  Therefore the same reasoning as for points of type (1) applies.
\end{remark}

\begin{proposition}\label{lemma:zero-convex}
Assume OBJ.  Then the optimization problem $F_{\mu}(x,\lambda,\theta)=0$ is zero-convex.
\end{proposition}

\begin{proof}
The convexity of the optimization problem at $\theta=0$ is implied by OBJ, and the parametric definitions of the hydraulic constraints (\ref{eq:discretized-mass-balance}) - (\ref{eq:discretized-momentum}),
which are linear for $\theta=0$.
\end{proof}

\begin{lemma}\label{lemma:span}
Assume BND, ICO, and HBC. Then the gradients of the hydraulic constraints (\ref{eq:discretized-mass-balance}) - (\ref{eq:discretized-momentum}) form a basis of the space of the interior $Q$ and $H$ variables.
\end{lemma}

\begin{proof}
Since the number of interior hydraulic variables equals the number of hydraulic constraints, the statement of the lemma is equivalent to showing that the gradients of the hydraulic constraints are linearly independent over the space of the interior $Q$ and $H$ variables; i.e., that the equation 
\begin{equation}\label{eq:sum}
\sum_{i \in I_H,\, j \in \{1,\ldots,T\}} \alpha_{i,j} \frac{\partial c_{i,j}}{\partial x_{hyd}} + \sum_{i \in I_Q,\, j \in \{1,\ldots,T\}} \beta_{i,j} \frac{\partial d_{i,j}}{\partial x_{hyd}}= 0
\end{equation}
is satisfied only if $\alpha_{i,j}=\beta_{i,j}=0$ for all $i \in I_H$ and $i \in I_Q$, respectively, and all $j \in \{1,\ldots, T\}$.

Let $x_{i,j}$ denote the subvector of the interior variables at interior discretization point $i$ and time step $t_j$, i.e., 
$$
x_{i,j} := \left( H_{i}(t_{j}),H_{i+1}(t_{j}), Q_{i-1}(t_{j}), Q_{i}(t_{j}) \right).
$$

For the sake of our proof, we only need to consider the partial derivatives of $c_{i,j}$ and $d_{i,j}$ restricted to $x_{i,j}$. The only non-zero terms are then equal to:
\begin{align*}
\frac{\partial c_{i,j}}{\partial x_{hyd}}\biggr\rvert_{x_{i,j}}
& = \left( \phi_{i,j}, 0, \frac{-1}{\Delta x},\frac{1}{\Delta x}\right); \\
\frac{\partial d_{i,j}}{\partial x_{hyd}}\biggr\rvert_{x_{i,j}}
 &= \left( -\psi_{i,j}, \psi_{i,j}, 0, \tau_{i,j} \right),
\end{align*}
where the vertical bar indicates restriction to a subvector, and where
\begin{align*}
\phi_{i,j} = & \frac{1}{\Delta t} \left(\theta \frac{\partial A_i}{\partial H_i}(H_i(t_{j})) + (1 - \theta) \overline{w} \right); \\
\psi_{i,j}  =  & \frac{g}{\Delta x} \left(\theta A_{i+\frac{1}{2}}(t_{j-1}) + (1 - \theta) \overline A \right); \\
\tau_{i,j} = & \frac{1}{\Delta t} + \theta  \frac{g}{C_i^2} \frac{P_{i+\frac{1}{2}}(t_{j-1}) \abs Q_i(t_{j-1})}{A_{i+\frac{1}{2}}(t_{j-1})^2}
 + (1-\theta) \frac{g}{C_i^2} \frac{\overline{P} \abs \overline{Q}}{\overline{A}^2}.
\end{align*}
We constructed the functions $A$ and $P$ such that $A > 0$, $\partial A/\partial H > 0$, and $P > 0$ for every $H$.  Furthermore, $\abs Q > 0$ for all $Q$.  Therefore the terms $\phi_{i,j}$, $\psi_{i,j}$ and $\tau_{i,j}$ must be nonzero for every time step $j~\in~\{1,\ldots, T\}$ and for any $i \in I_Q$.


We now proceed to prove that equation (\ref{eq:sum}) admits a unique solution. We first illustrate the reasoning for the simpler case when $T = 1$. For this we want to show that equation
\begin{equation}\label{eq:sum1}
\sum_{i \in I_H} \alpha_{i} \frac{\partial c_{i,1}}{\partial x_{hyd}} + \sum_{i \in I_Q} \beta_{i} \frac{\partial d_{i,1}}{\partial x_{hyd}}= 0,
\end{equation}
is satisfied only if $\alpha_{i}=\beta_{i}=0$ for all $i \in I_H$ and $i \in I_Q$, respectively.

Consider the $(|I_H| + |I_Q|)$-square matrix $M$ obtained by stacking on top of each other the gradients of the hydraulic constraints; i.e., the matrix whose rows are the gradients of the hydraulic constraints. 
The columns of $M$
are indexed by the interior variables $x_{hyd}$ and Equation~(\ref{eq:sum1}) has a unique solution if and only if the rows of $M$ are linearly independent.
When permuting rows such that continuity and momentum equations alternate, and columns such that flow and level variables alternate, $M$ is a tridiagonal matrix having the property that the $(m,n)$-entry of this matrix is nonzero if and only if $|m - n| \leq 1$.
Clearly, such a matrix has full rank and thus its rows are linearly  independent. As stated previously, this is equivalent to showing that Equation~(\ref{eq:sum1}) has a unique solution.



We will now prove the more general statement. 
Equation~(\ref{eq:sum}) is satisfied only if it holds even when we consider only part of its variables; i.e.,  equation 
\begin{align}\label{eq:sumr}
& \sum_{i \in I_H,\, j \in \{1,\ldots,T\}} \alpha_{i,j} \frac{\partial c_{i,j}}{\partial x_{hyd}}\biggr\rvert_{\widetilde x} 
+ \sum_{i \in I_Q,\, j \in \{1,\ldots,T\}} \beta_{i,j} \frac{\partial d_{i,j}}{\partial x_{hyd}}\biggr\rvert_{\widetilde x}= 0
\end{align}
holds for any subvector $\widetilde x$ of $x_{hyd}$.
We will use this simple observation to argue about the $\alpha$'s and $\beta$'s in Equation~(\ref{eq:sum}).

Let $x_T$ denote the subvector of $x_{hyd}$ that contains all the  internal hydraulic variables at time step $T$, i.e., the variables $H_i(t_T), Q_i(t_T)$ for $i \in I_H$ and $i \in I_Q$, respectively. 
As the variables of $x_T$ appear only in the gradients of the constraints $c_{i,T}$, $d_{i,T}$ for $i \in I_H$ and $i \in I_Q$, respectively, the following holds:
\begin{align*}
0 = & \sum_{i \in I_H,\, j \in \{1,\ldots,T\}} \alpha_{i,j} \frac{\partial c_{i,j}}{\partial x_{hyd}}\biggr\rvert_{x_T} 
+ \sum_{i \in I_Q,\, j \in \{1,\ldots,T\}} \beta_{i,j} \frac{\partial d_{i,j}}{\partial x_{hyd}}\biggr\rvert_{x_T} \\
= & \sum_{i \in I_H} \alpha_{i,T} \frac{\partial c_{i,T}}{\partial x_{hyd}}\biggr\rvert_{x_T} + \sum_{i \in I_Q} \beta_{i,T} \frac{\partial d_{i,T}}{\partial x_{hyd}}\biggr\rvert_{x_T}.
\end{align*}
We claim that the above equation has a solution only when all the $\alpha_{i,T}$ and the $\beta_{i,T}$ are equal to zero.
Let $M_T$ be the matrix whose rows are the gradients of the hydraulic constraints $\partial c_{i,T} \rvert_{x_T}$, $\partial d_{i,T} \rvert_{x_T}$ for $i \in I_H$ and $i \in I_Q$, respectively, restricted to the variables $x_T$. That is, $M_T$ is a $(|I_H| + |I_Q|)$-square matrix. 
When permuting rows such that continuity and momentum equations alternate, and columns such that flow and level variables alternate, $M_T$ is a tridiagonal matrix whose ($m,n$)-entry is non-zero if and only if $|m-n| \leq 1$ and, hence, the rows of 
$M_T$ are linearly independent. By construction, this is equivalent to saying that $\alpha_{i,T} = \beta_{i,T} = 0$ for all $i \in I_H$ and $i \in I_Q$, respectively. 

Let $x_{T-1}$ denote the subvector of $x_{hyd}$ that contains all the internal hydraulic variables at time step $T-1$, i.e., the variables $H_i(t_{T-1}), Q_i(t_{T-1})$ for $i \in I_H$ and $i \in I_Q$, respectively. 
Using the fact that the variables of $x_{T-1}$ appear only in the gradient of the constraints $\partial c_{i,T-1},\partial c_{i,T}, \partial d_{i,T-1},\partial d_{i,T}$ for $i \in I_H$ and $i \in I_Q$, respectively, and that all the $\alpha_{i,T}$ and the $\beta_{i,T}$ are equal to zero, we have that:
\begin{align*}
0 = & \sum_{i \in I_H,\, j \in \{1,\ldots,T\}} \alpha_{i,j} \frac{\partial c_{i,j}}{\partial x_{hyd}}\biggr\rvert_{x_{T-1}}
+ \sum_{i \in I_Q,\, j \in \{1,\ldots,T\}} \beta_{i,j} \frac{\partial d_{i,j}}{\partial x_{hyd}}\biggr\rvert_{x_{T-1}} \\
= & \sum_{i \in I_H} \alpha_{i,{T-1}} \frac{\partial c_{i,{T-1}}}{\partial x_{hyd}}\biggr\rvert_{x_{T-1}} 
+ \sum_{i \in I_Q} \beta_{i,{T-1}} \frac{\partial d_{i,{T-1}}}{\partial x_{hyd}}\biggr\rvert_{x_{T-1}}.
\end{align*}
By looking at the square matrix $M_{T-1}$ whose rows are the gradients of the constraints $\partial c_{i,T-1}, \partial d_{i,T-1}$ restricted to the variables $x_{T-1}$, 
we can use the same argument as before to show that all the $\alpha_{i,{T-1}}$ and the $\beta_{i,{T-1}}$ must be equal to zero. 

Repeating the reasoning when considering Equation~(\ref{eq:sumr}) for the internal hydraulic variables at time step $T-2$, then  $T-3$ and so on, we can conclude that all the $\alpha_{i,j}$ and $\beta_{i,j}$  in equation (\ref{eq:sum}) must be equal to zero.  Here, we have used the fact
that the initial condition is fully specified (ICO).  This concludes the proof.
\end{proof}

\begin{remark}
Lemma~\ref{lemma:span} crucially depends on the semi-implicit discretization of the pressure and friction terms.  If the these terms were discretized fully implicitly,
then particular combinations of $H$ and $Q$ would lead to vanishing $\psi_{i,j}$ and hence to singular points.
In the same vein, vanishing $\psi_{i,j}$ and $\tau_{i,j}$ introduce singular points when
considering a time-implicit discretization of the convective acceleration term.
\end{remark}

\begin{remark}
Other equations that relate a headloss $\Delta H_i(t_{j}):=H_{i}(t_{j}) - H_{i+1}(t_{j})$ to a flow $Q_i(t_{j})$ may be used instead of the momentum equation (\ref{eq:discretized-momentum}),
as long as the partial derivatives to $H_{i}(t_j)$, $H_{i+1}(t_j)$, and $H_i(t_j)$, and to no other interior hydraulic variables at time step $t_j$, are nonzero.  This covers typical hydraulic structures such as pumps, turbines, and weirs, as long 
as the headloss across the structure is nonzero.  Consider, for example, the semi-implicit version of the hydroelectric turbine equation
\[
P_i(t_j)=g \rho \eta\left(Q_i(t_{j-1}), H_i(t_{j-1}), H_{i+1}(t_{j-1})\right) Q_i(t_j) \Delta H_i(t_j)
\]
with instantaneous generation $P_i(t_j)$, gravitational constant $g$, density $\rho$, and efficiency function $\eta: \mathbb{R}^3 \to (0,1]$.  Following the notation of Lemma~\ref{lemma:span}, 
\[
\psi_{i,j}=-g\rho \eta\left(Q_i(t_{j-1}), H_i(t_{j-1}), H_{i+1}(t_{j-1})\right) Q_i(t_j),
\]
and
\[
\tau_{i,j}=g\rho \eta\left(Q_i(t_{j-1}), H_i(t_{j-1}), H_{i+1}(t_{j-1})\right) \Delta H_i(t_j). 
\]
These are nonzero as long as $Q_i(t_j) \neq 0$ and $\Delta H_i(t_j) \neq 0$, i.e., as long as the structure is operating.  When $Q_i(t_j)=0$, the equation collapses to an interior boundary condition.
The condition $\Delta H_i(t_j)=0$ may occur in the context of pumps and weirs and in that case, a boolean switching variable needs to be introduced as in \cite{baayen2020mixed}.
This illustrates how convex relaxations of the turbine and pump equations \cite[e.g]{horvath2019convex} may be avoided, and how the exact machine characteristics may be used instead.
\end{remark}

\begin{corollary}\label{lemma:path-connected}
Assume BND, ICO, HBC, and QBC. 
Then, for any $\theta \in [0,1]$, the set of feasible interior points 
is non-empty and path-connected.
\end{corollary}

\begin{proof}
Fix $\theta \in [0,1]$. Partition the vector $x$ into interior hydraulic and boundary components,
\[
x:=(x_{hyd},x_{bdy}).
\]
For any $x_{bdy}$, we may integrate the dynamics forwards in time, starting from the initial conditions (ICO).
In \cite{casulli2009high}, Theorem 1 and 2, Casulli shows that the solution of the equations (\ref{eq:discretized-mass-balance}) and
(\ref{eq:discretized-momentum}) exists and is unique.  Note here that our construction is one-dimensional, 
so that for Casulli's element volume $V$ we use 
\[
V=\Delta x\left(\theta A + (1-\theta) \overline{w}H\right).
\]
We also do not treat wetting and drying in same way. By our construction $A > 0$ and $\partial A / \partial H > 0$ for all $H$, so that Casulli's Theorem 1 always holds.
Hence, we have a unique map $g: x_{bdy} \mapsto x_{hyd}$
that is defined for all $x_{bdy}$.
By Lemma \ref{lemma:span} and the implicit function theorem \cite{rudin1964principles}, this function $g$ is locally continuous,
hence continuous everywhere\footnote{If we would have used a non-mass-conservative discretization of the continuity equation, with $\partial A/\partial t$ discretized as $\partial A/\partial H\left(H(t_{j-1})\right) \cdot \left(H(t_j) - H(t_{j-1})\right)/\Delta t$, 
the equations (\ref{eq:discretized-mass-balance}) - 
(\ref{eq:discretized-momentum})  would both be linear in the variables at time step $t_j$.  Then, existence, uniqueness, and continuity would follow directly from
Lemma \ref{lemma:span} and the continuity of the matrix inverse in the matrix coefficients.}.

Finally, note that the set of interior boundary variables is convex, hence path-connected.
Therefore the image under $g$ is also path-connected.   Since the interior hydraulic variables are unbounded (BND),
the set of feasible interior points is the direct sum of the boundary variables and their image under $g$. 
Since both sets are path-connected, the set of feasible interior points is also path-connected.
\end{proof}

\begin{lemma}\label{lemma:hessian2}
Assume BND, HBC, and OBJ.
Partition the vector $x$ into interior hydraulic and boundary components,
\[
x:=(x_{hyd},x_{bdy}).
\]
Then the Hessian of the Lagrangian with respect to $x_{bdy}$, $\nabla_{x_{bdy}x_{bdy}}^2 \mathcal{L}_{\mu}(x,\lambda, \theta)$, is positive definite.
\end{lemma}
\begin{proof}
Since we assume all boundary variables to be bounded (BND), the second derivatives of the logarithmic barrier functions with respect to $x_{bdy}$ form a positive definite diagonal matrix.  To this we add the Hessian of the objective function $f$
with respect to $x_{bdy}$, which is positive semi-definite due to the convexity of $f$ (OBJ). 

As the sign of the Lagrange multipliers is not known a-priori, the Hessians of the constraints may have an indefinite contribution to the Hessian of the Lagrangian, potentially resulting in a loss of positive semi-definiteness.
This, however, can only happen for constraints that are nonlinear in $x_{bdy}$. As per (HBC), $H$ boundaries are fixed and hence do not occur as optimization variables.  Free $Q$ boundaries do occur, but only the mass balance equation (\ref{eq:discretized-mass-balance}) and the convective acceleration term in the momentum equation (\ref{eq:discretized-momentum}) depend on boundary $Q$ variables.  Both the mass balance equation and the convective acceleration term are linear in the boundary flow variables, and therefore do not contribute to $\nabla_{x_{bdy}x_{bdy}}^2 \mathcal{L}_{\mu}$.

Hence, the Hessian matrix $\nabla_{x_{bdy}x_{bdy}}^2 \mathcal{L}_{\mu}$ is positive definite.
\end{proof}

\begin{lemma}\label{lemma:hessian3}
Assume BND, HBC, and OBJ. Then the Hessian of the Lagrangian with respect to $x$, $\nabla_{xx}^2 \mathcal{L}_{\mu}(x,\lambda, \theta)$, is nonsingular
on the tangent space $T(x,\theta)$.
\end{lemma}
\begin{proof}
Fix $\mu > 0$ and $\theta \in [0,1]$.  Partition the vector $x$ into interior hydraulic and boundary components,
\begin{equation}\label{eq:partition1}
x:=(x_{hyd},x_{bdy}).
\end{equation}
With respect to the partitioning (\ref{eq:partition1}) the Hessian matrix $\nabla_{xx}^2 \mathcal{L}_{\mu}(x,\lambda, \theta)$ has the block form
\[
\begin{pmatrix}
\nabla_{x_{hyd}x_{hyd}}^2 \mathcal{L}_{\mu} & \nabla_{x_{hyd}x_{bdy}}^2 \mathcal{L}_{\mu} \\
\nabla_{x_{bdy}x_{hyd}}^2 \mathcal{L}_{\mu} & \nabla_{x_{bdy}x_{bdy}}^2 \mathcal{L}_{\mu}
\end{pmatrix}.
\]
By Lemma \ref{lemma:hessian2}, the submatrix $\nabla_{x_{bdy}x_{bdy}}^2 \mathcal{L}_{\mu}$ is positive definite,
hence nonsingular.  Using elementary row and column operations multiplied into a matrix $P$, we can transform the Hessian to have the form
\[
P \nabla_{xx}^2 \mathcal{L}_{\mu}(x,\lambda, \theta) P =
\begin{pmatrix}
A & 0 \\
0 & \nabla_{x_{bdy}x_{bdy}}^2 \mathcal{L}_{\mu}
\end{pmatrix}
\]
for some square matrix $A$.  Now, for any element of the tangent space $0 \neq y \in T(x,\theta)$, we have
$$
\nabla_{xx}^2 \mathcal{L}_{\mu}(x,\lambda, \theta) y = P^{-1} (P \nabla_{xx}^2 \mathcal{L}_{\mu}(x,\lambda, \theta) P) P^{-1} y.
$$
Since $P$ only modifies the rows/columns corresponding to the interior hydraulic variables, $P^{-1} y = y$ on the space
of boundary variables.  Since the submatrix
$\nabla_{x_{bdy}x_{bdy}}^2 \mathcal{L}_{\mu}$ is positive definite by Lemma \ref{lemma:hessian2},  it remains to show that $y$ must have at least one nonzero boundary variable coordinate.
This must be so by Lemma \ref{lemma:span}, for otherwise $y$ would not be orthogonal to all constraint gradients.
We have shown that the matrix $\nabla_{xx}^2 \mathcal{L}_{\mu}$ is nonsingular on the tangent space $T(x,\theta)$.
\end{proof}

\begin{proposition}\label{thm:main-result}
Assume BND, ICO, HBC, QBC, and OBJ.
Then the optimization problem $F_{\mu}(x,\lambda,\theta)=0$ is path-stable with respect to the interior point method.
\end{proposition}

\begin{proof}
The result follows directly from an application of Proposition \ref{theorem:tiahrt}, noting a) that Lemma \ref{lemma:hessian3}
proves the nonsingularity of the Hessian of the Lagrangian on the tangent space of the constraint manifold, and b) that Lemma \ref{lemma:span}
proves the linear independence of the constraint gradients.
\end{proof}

Combining the above results for zero-convexity, path-stability, and the path-connectedness of the search space,
we obtain the following uniqueness theorem for optimization problems constrained by the shallow water equations:

\begin{theorem}\label{thm}
Assume BND, ICO, HBC, QBC, and OBJ.  Then the non-convex optimization problem $F_{\mu}(x,\lambda,\theta)=0$ has a unique solution for every $\mu > 0$ and every $\theta \in [0,1]$.
\end{theorem}

\begin{proof}
Existence follows from Corollary \ref{lemma:path-connected}.  Uniqueness follows from Propositions \ref{lemma:zero-convex} and \ref{thm:main-result}, Corollary \ref{lemma:path-connected}, and Theorem \ref{prop:uniqueness}.
\end{proof}

Theorem \ref{thm} implies that the optimization problem $F_{\mu}(x,\lambda,\theta)=0$ can be solved
to global optimality using a continuation method.

\subsection{Numerical example}
\label{sec:example}

We consider a single river reach with 10 uniformly spaced water level nodes and rectangular cross section, an upstream inflow boundary condition provided with a fixed time series, as well as a controllable downstream release boundary condition. 
The grid is illustrated in Figure~\ref{fig:example-grid}, and the hydraulic parameters and initial conditions are summarized in Table~\ref{table:example-specs}. The model starts from steady state: the initial flow rate is uniform and the water level decreases linearly along the length of the channel.

\begin{figure}[!t]
\centering
\begin{tikzpicture}[scale=0.8]
\draw (0,0) -- (10.0,0);
\foreach \x in {0,1,...,10}
    \draw (\x cm,1 pt) -- (\x cm,-1 pt) node[anchor=south] { \begin{tabular}{c} $Q_{\x}$ \end{tabular} };
\foreach \x in {1,2,...,10}
    \draw (\x cm - 0.5 cm,1 pt) -- (\x cm - 0.5 cm,-1 pt) node[anchor=north] {$H_{\x}$};
\end{tikzpicture}
\caption{Staggered grid for the example problem.}
\label{fig:example-grid}
\end{figure}
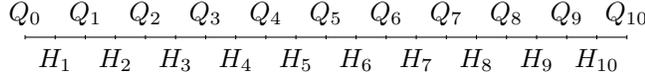

To give a physical context for this problem, suppose this model represents a channel downstream of a reservoir and upstream of an adjustable weir with limited capacity.
The weir is trying to dampen the sudden pulse of water shown in Figure~\ref{fig:results}(b) released by the reservoir.

\begin{table}[!t]
  \caption{Parameters for the example problem.}
  \label{table:example-specs}
  \centering
  \begin{tabularx}{\linewidth}{ l l X}
  \hline
     Parameter & Value & Description \\
  \hline
      $T$ & $72$ & Index of final time step \\
      $\Delta t$ & $600$ s & Time step size \\
      $H^b_{i}$ & $\left( -4.90, -4.92, \ldots, -5.10 \right)$ m & Bottom level\\
      $l$ & $10 000$ m & Total channel length \\
      $A_i(H_i)$ & $50\cdot(H-H^b_i)$ m$^2$ & Channel cross section function \\
      $P_i(H_i)$ & $50+2\cdot(H-H^b_i)$ m & Channel wetted perimeter function \\
      $C_{i}$ & $\left( 40, 40, \ldots, 40 \right)$ m$^{0.5}$/s & Chézy friction coefficient \\
      $\overline{H}$ & $0.0$ m & Nominal level in linear model for entire reach \\
      $\overline{Q}$ & $100$ m$^3$/s & Nominal discharge in linear model for entire reach \\
      $H_i(t_0)$ & $\left( 0.000, -0.025, \ldots, -0.222 \right)$ m & Initial water levels at $H$ nodes \\
      $Q_i(t_0)$ & $\left( 100, 100, \ldots, 100 \right)$ m$^3$/s & Initial discharge at $Q$ nodes \\
      $\varepsilon$ & $10^{-12}$ & Absolute value approximation smoothness parameter \\
      $K$ & $10$ & Convective acceleration steepness factor \\
    \hline
  \end{tabularx}
  
\end{table}

Our optimization objective is to keep the water level at the $H$ nodes at $0$ m above datum:
\[
\min \sum_{i=1}^{10} \sum_{j=1}^T H_i(t_j)^2
\]
subject to the adjustable weir flow constraint
\[
100\, \text{m}^3/\text{s} \leq Q_{10} \leq 200\, \text{m}^3/\text{s}.
\]

The solution to the optimization problem is plotted in Figure~\ref{fig:results}.
By releasing water in anticipation of the inflow using the decision variable $Q_{10}$, the optimization is able to reduce water level fluctuations and keep the water levels close to the target level.

This optimization problem was implemented in Python using the \textsc{CasADi} package \cite{andersson2019casadi} for algorithmic differentiation, and connected to the \textsc{IPOPT} optimization solver \cite{Wachter2006}.  
On a MacBook Pro with 2.9 GHz Intel Core i5 CPU, the example takes approximately $0.4$ s to solve.
The complete source code is available online at \\ \mbox{\url{https://github.com/jbaayen/homotopy-example}}.

\begin{figure}[!t]
\centering
 \includegraphics[scale=0.75]{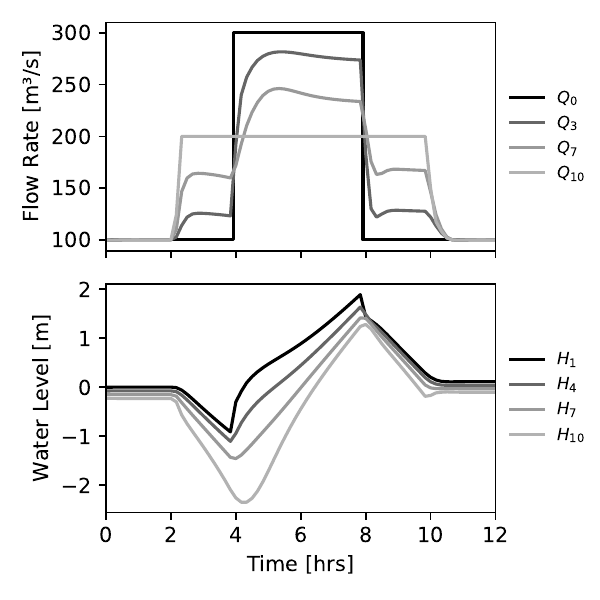}
\caption{Solution of the example optimization problem.}
\label{fig:results} 
\end{figure}

For a real-world case study that also includes other types of objective functions, we refer the reader to \cite{baayen2019overview}.
An extension to mixed-integer decision variables 
is covered in \cite{baayen2020mixed}.

For a numerical benchmark of larger problems, we refer the reader to \cite{Baayen2019-3}. This benchmark covers a comparison with a heuristic from the water resources management literature,
varying numbers of interconnected reaches (up to $16$), as well as varying numbers of discretization points (up to $512$ water level nodes), over an optimization horizon of $48$ hours with varying time step sizes (down to $5$ minutes).

\section{Conclusions}
In the first part of this paper, we provided sufficient conditions under which each local optimum of a non-convex optimization
problem is a global optimum.

The analysis rests on a path-stable, i.e., bifurcation-free, homotopy between a convex relaxation, and the original non-convex optimization problem.  
The path-stable homotopy transfers the non-negativity of the eigenvalues of the reduced Hessian
of the convex relaxation, to the non-convex problem.  In this way, the need to analyze the spectrum
of the reduced Hessian of the non-convex problem directly -- which is, in general, hard -- is bypassed. 

In the second part of the paper, an application is presented to a class of optimization problems subject to the shallow water equations,
that describes flow in rivers and canals.  It is illustrated
how also hydraulic structures, such as hydroelectric turbines in hydropower schemes, may be modelled non-linearly within 
the path-stable continuation framework.

Path-stable, zero-convex optimization problems may be solved to global optimality using local search.
This enables highly efficient numerical implementations, thereby rendering the approach suitable for closed-loop model predictive control
of large-scale cyber-physical systems. This is attested by its practical implementation for closed-loop
control of the primary waterways of the Rijnland water authority in the Netherlands.

\section{Acknowledgments}

The authors would like to thank Jan van Schuppen, Pierre Archambeau, Jakub Mareček, and Dirk Schwanenberg for their comments.

TKI Delta Technology provided part of the funding under projects DEL021 and DEL029.

\bibliographystyle{siamplain}
\bibliography{paper}

\end{document}